\documentclass{amsart}
\usepackage{amssymb}
\usepackage{amsfonts}
\usepackage{amssymb}
\usepackage{amsmath}
\usepackage{amsthm}
\usepackage{enumerate}
\usepackage{tabularx}
\usepackage{centernot}
\usepackage{mathtools}
\usepackage{stmaryrd}
\usepackage{amsthm,amssymb}
\usepackage{etoolbox}
\usepackage{tikz}
\usepackage{amssymb}
\usetikzlibrary{matrix}
\usepackage{tikz-cd}
\usepackage{tikz}
\usepackage{marginnote}
\definecolor{mygray}{gray}{0.85}
\usepackage[backgroundcolor=mygray,colorinlistoftodos,prependcaption,textsize=small]{todonotes}
\usepackage{xargs}                      

\renewcommand{\leq}{\leqslant}
\renewcommand{\geq}{\geqslant}

\makeatletter
\def\subsection{\@startsection{subsection}{3}%
  \z@{.5\linespacing\@plus.7\linespacing}{.3\linespacing}%
  {\bfseries\centering}}
\makeatother

\makeatletter
\def\subsubsection{\@startsection{subsubsection}{3}%
  \z@{.5\linespacing\@plus.7\linespacing}{.3\linespacing}%
  {\centering}}
\makeatother

\makeatletter
\def\myfnt{\ifx\protect\@typeset@protect\expandafter\footnote\else\expandafter\@gobble\fi}
\makeatother

\newtheorem{theorem}{Theorem}

\newtheorem{corollary}[theorem]{Corollary}
\newtheorem{definition}[theorem]{Definition}
\newtheorem{lemma}[theorem]{Lemma}
\newtheorem{proposition}[theorem]{Proposition}

\newtheorem{problem}[theorem]{Problem}

\newtheorem{observation}[theorem]{Observation}
\newtheorem{fact}[theorem]{Fact}

\newtheorem{convention}[theorem]{Convention}
\newtheorem{notation}[theorem]{Notation}

\newcounter{claimcounter}
\numberwithin{claimcounter}{theorem}
\newenvironment{claim}{\stepcounter{claimcounter}{\noindent {\underline{\em Claim \theclaimcounter}.}}}{}
\newenvironment{claimproof}[1]{\noindent{{\em Proof.}}\space#1}{\hfill $\rule{0.40em}{0.40em}$}

\begin{document}

\begin{abstract} We give a complete characterization of the graph products of cyclic groups admitting a Polish group topology, and show that they are all realizable as the group of automorphisms of a countable structure. In particular, we characterize the right-angled Coxeter groups (resp. Artin groups) admitting a Polish group topology. This generalizes results from \cite{shelah}, \cite{shelah_1} and  \cite{paolini&shelah}.
\end{abstract}

\title{Polish Topologies for Graph Products of Cyclic Groups}
\thanks{Partially supported by European Research Council grant 338821. No. 1115 on Shelah's publication list.}

\author{Gianluca Paolini}
\address{Einstein Institute of Mathematics,  The Hebrew University of Jerusalem, Israel}

\author{Saharon Shelah}
\address{Einstein Institute of Mathematics,  The Hebrew University of Jerusalem, Israel \and Department of Mathematics,  Rutgers University, U.S.A.}

\date{\today}
\maketitle


\section{Introduction}

	\begin{definition}\label{def_cyclic_prod} Let $\Gamma = (V, E)$ be a graph and $\mathfrak{p}: V \rightarrow \{ p^n : p \text{ prime and } 1 \leq n \} \cup \{ \infty \}$ a graph colouring. We define a group $G(\Gamma, \mathfrak{p})$ with the following presentation:
	$$ \langle V \mid a^{\mathfrak{p}(a)} = 1, \; bc = cb : \mathfrak{p}(a) \neq \infty \text{ and }  b E c \rangle.$$
\end{definition}
We call the group $G(\Gamma, \mathfrak{p})$ the {\em $\Gamma$-product\footnote{Notice that this is consistent with the general definition of graph products of groups from \cite{green}. In fact every graph product of cyclic groups can be represented as $G(\Gamma, \mathfrak{p})$ for some $\Gamma$ and $\mathfrak{p}$ as above.} of the cyclic groups} $\{ C_{\mathfrak{p}(v)} : v \in \Gamma \}$, or simply the {\em graph product of} $(\Gamma, \mathfrak{p})$. The groups $G(\Gamma, \mathfrak{p})$ where $\mathfrak{p}$ is constant of value $\infty$ (resp. of value $2$) are known as {\em right-angled Artin groups} $A(\Gamma)$ (resp. {\em right-angled Coxeter groups} $C(\Gamma)$). These groups have received much attention in combinatorial and geometric group theory. In the present paper we tackle the following problem:

	\begin{problem}\label{problem} Characterize the graph products of cyclic groups admitting a Polish group topology, and which among these are realizable as the group of automorphisms of a countable structure.
\end{problem}

	This problem is motivated by the work of Shelah \cite{shelah} and Solecki \cite{solecki}, who showed that no uncountable Polish group can be free or free abelian (notice that for $\Gamma$ discrete (resp. complete) $A(\Gamma)$ is a free group (resp. a free abelian group)). These negative results have been later generalized by the authors to the class of uncountable right-angled Artin groups \cite{paolini&shelah}. In this paper we give a complete solution to Problem \ref{problem} proving the following theorem:

	\begin{theorem}\label{main_th} Let $G = G(\Gamma, \mathfrak{p})$, and recall that $\mathfrak{p}$ is a graph colouring (cf. Definition \ref{def_cyclic_prod}), and so we refer to the elements in the range of $\mathfrak{p}$ as colors. Then $G$ admits a Polish group topology if only if $(\Gamma, \mathfrak{p})$ satisfies the following four conditions:
	\begin{enumerate}[(a)]
	\item there exists a countable $A \subseteq \Gamma$ such that for every $a \in \Gamma$ and $a \neq b \in \Gamma - A$, $a$ is adjacent to $b$;
	\item there are only finitely many colors $c$ such that the set of vertices of color $c$ is uncountable;
	\item there are only countably many vertices of color $\infty$;
	\item if there are uncountably many vertices of color $c$, then the set of vertices of color $c$ has the size of the continuum.
\end{enumerate}
Furthermore, if $(\Gamma, \mathfrak{p})$ satisfies conditions (a)-(d) above, then $G$ can be realized as the group of automorphisms of a countable structure.
\end{theorem}

	Thus, the only graph products of cyclic groups admitting a Polish group topology are the direct sums $G_1 \oplus G_2$ with $G_1$ a countable graph product of cyclic groups and $G_2$ a direct sum of finitely many continuum sized vector spaces over a finite field. From our general result we deduce a solution to Problem \ref{problem} in the particular case of right-angled Artin groups (already proved in \cite{paolini&shelah}) \mbox{and right-angled Coxeter groups.}

	\begin{corollary} No uncountable Polish group can be a right-angled Artin group.
\end{corollary}

\begin{corollary} An uncountable right-angled Coxeter group $C(\Gamma)$  admits a Polish group topology if and only if it is realizable as the group of automorphisms of a countable structure if and only if $|\Gamma| = 2^{\omega}$ and there exists  a countable $A \subseteq \Gamma$ such that for every $a \in \Gamma$ and $a \neq b \in \Gamma - A$, $a$ is adjacent to $b$.
\end{corollary}

	In works in preparation we deal with the characterization problem faced here in the more general setting of graph products of general groups \cite{paolini&shelah2}, and with questions of embeddability of graph products of groups into Polish groups \cite{paolini&shelah1}.

\section{Preliminaries}


	We will make a crucial use of the following special case of \cite[3.1]{shelah_1}.
	
	\begin{notation} By a group term $\sigma(\bar{x})$ we mean a word in the alphabet $\{ x : x \in \bar{x} \}$, i.e. an expression of the form $x_{1}^{\varepsilon_{1}} \cdots x_{n}^{\varepsilon_{n}}$, 
where $x_1,..., x_n$ are from $\bar{x}$ and each $\varepsilon_i$ is either $1$ or $-1$. The number $n$ is known as the length of the group term $\sigma(\bar{x})$.
\end{notation}
	
	\begin{fact}[\cite{shelah_1}]\label{771_fact} Let $G = (G, \mathfrak{d})$ be a Polish group and $\bar{g} = (\bar{g}_n : n < \omega)$, with $\bar{g}_n \in G^{\ell(n)}$ and $\ell(n) < \omega$.
	\begin{enumerate}[(1)]
	\item For every non-decreasing $f \in \omega^\omega$ with $f(n) \geq 1$ and $(\varepsilon_n)_{n < \omega} \in (0, 1)^{\omega}_{\mathbb{R}}$ there is a sequence $(\zeta_n)_{n < \omega}$ (which we call an $f$-continuity sequence for $(G, \mathfrak{d}, \bar{g})$, or simply an $f$-continuity sequence) satisfying the following conditions:
	\begin{enumerate}[(A)]
	\item for every $n < \omega$:
	\begin{enumerate}[(a)]
	\item $\zeta_n \in (0, 1)_{\mathbb{R}}$ and $\zeta_n < \varepsilon_n$;
	\item $\zeta_{n+1} < \zeta_{n}/2$;
	\end{enumerate}
	\end{enumerate}
	\begin{enumerate}[(B)]
	\item for every $n < \omega$, group term $\sigma(x_0, ..., x_{m-1}, \bar{y}_n)$ and $(h_{(\ell, 1)})_{\ell < m}, (h_{(\ell, 2)})_{\ell < m} \in G^m$, the $\mathfrak{d}$-distance from $\sigma(h_{(0, 1)}, ..., h_{(m-1, 1)}, \bar{g}_n)$ to $\sigma(h_{(0, 2)}, ..., h_{(m-1, 2)}, \bar{g}_n)$ is $< \zeta_n$, when: 
	\begin{enumerate}[(a)]
	\item $m \leq n+1$;
	\item $\sigma(x_0, ..., x_{m-1}, \bar{y}_n)$ has length $\leq f(n)+1$;
	\item $h_{(\ell, 1)}, h_{(\ell, 2)} \in Ball(e; \zeta_{n+1})$;
	\item $G \models \sigma(e, ..., e, \bar{g}_n) = e$.
	\end{enumerate}
	\end{enumerate}
	\item The set of equations $\Gamma = \{ x_n = (x_{n+1})^{k(n)} d_{n} : n < \omega \}$ is solvable in $G$ when for every $n < \omega$:
	\begin{enumerate}[(a)] 
	\item $f \in \omega^\omega$ is non-decreasing and $f(n) \geq 1$;
	\item $1 \leq k(n) < f(n)$;
	\item $(\zeta_n)_{n < \omega}$ is an $f$-continuity sequence;
	\item $\mathfrak{d}(d_{n}, e) < \zeta_{n+1}$.
	\end{enumerate}
	\end{enumerate}
\end{fact}

	\begin{convention}\label{convention} If we apply Fact \ref{771_fact}(1) without mentioning $\bar{g}$ it means that we apply Fact \ref{771_fact}(1) for $\bar{g}_n = \emptyset$, for every $n < \omega$.
\end{convention}


We shall use the following observation freely throughout the paper.

	\begin{observation}\label{observation_prelim} Suppose that $(G, \mathfrak{d})$ is Polish, $A \subseteq G$ is uncountable and $\zeta > 0$. Then for some $g_1 \neq g_2 \in A$ we have $\mathfrak{d}((g_1)^{-1}g_2, e) < \zeta$. 
\end{observation}

	\begin{proof} First of all, notice that we can find $g_1 \in A$ such that $g_1$ is an accumulation point of $A$, because otherwise we contradict the separability of $(G, \mathfrak{d})$. Furthermore, the function $(x, y) \mapsto x^{-1}y$ is continuous and so for every $(x_1, y_1) \in G^2$ and $\zeta > 0$ there is $\delta > 0$ such that, for every $(x_2, y_2) \in G^2$, if $\mathfrak{d}(x_1, x_2), \mathfrak{d}(y_1, y_2) < \delta$, then $\mathfrak{d}((x_1)^{-1}y_1, (x_2)^{-1}y_2) < \zeta$. Let now $g_2 \in Ball(g_1; \delta) \cap A - \{g_1\}$, then $\mathfrak{d}((g_1)^{-1}g_2, (g_1)^{-1}g_1) = \mathfrak{d}((g_1)^{-1}g_2, e) < \zeta$, and so we are done.
\end{proof}

	Before proving Lemma \ref{lemma1} we need some preliminary work. Given $A \subseteq \Gamma$ we denote the induced subgraph of $\Gamma$ on vertex set $A$ as $\Gamma_A$.

	\begin{fact}\label{fact} Let $G = G(\Gamma, \mathfrak{p})$, $A \subseteq \Gamma$ and $G_A = (\Gamma_A, \mathfrak{p} \restriction A)$. Then there exists a unique homomorphism $\mathbf{p} = \mathbf{p}_A: G \rightarrow G_A$ such that $\mathbf{p}(c) = c$ if $c \in A$, and $\mathbf{p}(c) = e$ if $c \notin A$.
\end{fact}

	\begin{proof} For arbitrary $G = G(\Gamma, \mathfrak{p})$, let $\Omega_{(\Gamma, \mathfrak{p})}$ be the set of equations from Definition \ref{def_cyclic_prod} defining $G(\Gamma, \mathfrak{p})$. Then for the $\Omega_{(\Gamma, \mathfrak{p})}$ of the statement of the fact we have $\Omega_{(\Gamma, \mathfrak{p})} = \Omega_1 \cup \Omega_2 \cup \Omega_3$, where:
	\begin{enumerate}[(a)]
	\item $\Omega_1 = \Omega_{(\Gamma_A, \mathfrak{p} \restriction A)}$;
	\item $\Omega_2 = \Omega_{(\Gamma_{\Gamma - A}, \mathfrak{p} \restriction \Gamma - A})$;
	\item $\Omega_3 = \{ bc = cb : b E_{\Gamma} c \text{ and } \{ b, c \} \not\subseteq A \}$.
\end{enumerate}
Notice now that $\mathbf{p}$ maps each equation in $\Omega_1$ to itself and each equation in $\Omega_2 \cup \Omega_3$ to a trivial equation, and so $p$ is an homomorphism (clearly unique).
\end{proof}

	\begin{definition} Let $(\Gamma, \mathfrak{p})$ be as usual and $G = G(\Gamma, \mathfrak{p})$.
	\begin{enumerate}[(1)]
	\item A word $w$ in the alphabet $\Gamma$ is a sequence $(a_1^{\alpha_1}, ..., a_k^{\alpha_k})$, with $a_1 \neq a_2 \neq \cdots \neq a_k \in \Gamma$ and $\alpha_1, ..., \alpha_k \in \mathbb{Z} - \{0 \}$.
	\item We denote words simply as $a_1^{\alpha_1} \cdots a_k^{\alpha_k}$ instead of $(a_1^{\alpha_1}, ..., a_k^{\alpha_k})$.
	\item We call each $a_i^{\alpha_i}$ a syllable of the word $a_1^{\alpha_1} \cdots a_k^{\alpha_k}$.
	\item We say that the word $a_1^{\alpha_1} \cdots a_k^{\alpha_k}$ spells the element $g \in G$ if $ G \models g = a_1^{\alpha_1} \cdots a_k^{\alpha_k}$.
	\item We say that the word $w$ is reduced if there is no word with fewer syllables which spells the same element of $G$.
	\item We say that the consecutive syllables $a_i^{\alpha_i}$ and $a_{i+1}^{\alpha_{i+1}}$ are adjacent if $a_iE_{\Gamma}a_{i+1}$.
	\item We say that the word $w$ is a normal form for $g$ if it spells $g$ and it is reduced.
	\item We say that two normal forms are equivalent if there they spell the same element $g \in G$.
\end{enumerate}
\end{definition}

	As usual, when useful we identify words with the elements they spell.
	
	\begin{fact}[{\cite[Lemmas 2.2 and 2.3]{gut}}]\label{fact_word_1} Let $G = G(\Gamma, \mathfrak{p})$.
	\begin{enumerate}[(1)] 
	\item If the word $a_1^{\alpha_1} \cdots a_k^{\alpha_k}$ spelling the element $g \in G$ is not reduced, then there exist $1 \leq p < q \leq k$ such that $a_p = a_q$ and $a_p$ is adjacent to each vertex $a_{p + 1}, a_{p + 2}, ..., a_{q-1}$.
	\item If $w_1 = a_1^{\alpha_1} \cdots a_k^{\alpha_k}$ and $w_2 = b_1^{\beta_1} \cdots b_k^{\beta_k}$ are normal forms for $g \in G$, then $w_1$ can be transformed into $w_2$ by repetedly swapping the order of adjacent syllables.
\end{enumerate} 
\end{fact}


	\begin{definition} Let $g \in G(\Gamma, \mathfrak{p})$. We define:
	\begin{enumerate}[(1)]
	\item $sp(g) = \{ a_i \in \Gamma : a_1^{\alpha_1} \cdots a_i^{\alpha_i} \cdots a_k^{\alpha_k} \text{ is a normal form for } g \}$;
	\item $F(g) = \{ a_1^{\alpha_1} : a_1^{\alpha_1} \cdots a_k^{\alpha_k} \text{ is a normal form for } g\}$;
	\item $L(g) = \{ a_k^{\alpha_k} : a_1^{\alpha_1} \cdots a_k^{\alpha_k} \text{ is a normal form for } g\}$;
	\item $\hat{L}(g) = \{ a_k^{-\alpha_k}:  a_k^{\alpha_k} \in L(g) \}$.
\end{enumerate}
\end{definition}

\begin{definition} We say that the normal form $a_1^{\alpha_1} \cdots a_k^{\alpha_k}$ is cyclically normal if either $k = 1$ or there is no equivalent normal form $b_1^{\beta_1} \cdots b_k^{\beta_k}$ with $b_1 = b_k$.
\end{definition}

\begin{observation}
\begin{enumerate}[(1)]
\item Notice that if $g \in G(\Gamma, G_a)$ is spelled by a cyclically normal form, then any of the normal forms spelling $g$ are cyclically normal.
\item We say that the group element $g \in G(\Gamma, G_a)$ is cyclically normal if any of the normal forms (which are words) spelling $g$ are cyclically normal.
\end{enumerate}
\end{observation}

	\begin{notation}\label{notation} Given a sequence of words $w_1, ..., w_k$ with some of them possibly empty, we say that the word $w_1 \cdots w_k$ is a normal form (resp. a cyclically normal form) if after deleting the empty words the resulting word is a normal form (resp. a cyclically normal form). 
\end{notation}
	
	Recall that given $A \subseteq \Gamma$ we denote the induced subgraph of $\Gamma$ on vertex set $A$ as $\Gamma_A$.

	\begin{fact}[{\cite[Corollary 24]{bark}}]\label{bark_fact} Any element $g \in G(\Gamma, \mathfrak{p})$ can be written in the form $w_1 w_2 w_3 w'_2 w^{-1}_1$, where:
	\begin{enumerate}[(1)]
	\item $w_1 w_2 w_3 w'_2 w^{-1}_1$ is a normal form;
	\item $w_3 w'_2 w_2$ is cyclically normal;
	\item $sp(w_2) = sp(w'_2)$;
	\item if $w_2 \neq e$, then $\Gamma_{sp(w_2)}$ is a complete graph;
	\item $F(w_2) \cap \hat{L}(w'_2) = \emptyset$.
	\end{enumerate}
\end{fact}

%
	
	\begin{proposition}\label{fact_word} Let $G = G(\Gamma, \mathfrak{p})$, and assume that $\mathfrak{p}$ has finite range $\{ c_1, ..., c_t \}$. Let $p$ be a prime such that if $c_i \neq \infty$ then $p > c_i$, for $i =1, ..., t$. Then for every $g \in G$ we have $sp(g) \subseteq sp(g^{p})$.
\end{proposition}

	\begin{proof} Let $g$ be written as $w_1 w_2 w_3 w'_2 w^{-1}_1$ as in Fact \ref{bark_fact}, and assume $g \neq e$. We make a case distinction.
\newline \underline{\em Case 1}. $w_3 = e$.
\newline Notice that $w_2 w'_2 \neq e$, because by assumption $g \neq e$, and that $w_2 w'_2$ is a normal form (recall Notation \ref{notation}). Let $a_1^{\alpha_1} \cdots a_k^{\alpha_k}$ be a normal form for $w_2 w'_2$. Then by items (3) and (4) of Fact \ref{bark_fact} we have:
$$g^{p} = w_1 (a_1^{\alpha_1} \cdots a_k^{\alpha_k})^{p} w^{-1}_1 = w_1 a_1^{p\alpha_{1}} \cdots a_k^{p\alpha_k} w^{-1}_1.$$
Now, necessarily, for every $\ell \in \{ 1, ..., k \}$, $a_{\ell}^{p \alpha_\ell} \neq e$,  since the order of $a_{\ell}$ does not divide $\alpha_\ell$ and $p$ is a prime. Thus, we are done.
\newline \underline{\em Case 2}. $w_2 = e$.
\newline By item (3) of Fact \ref{bark_fact} also $w'_2 = e$, and so, by item (2) of Fact \ref{bark_fact}, $w_3 w'_2 w_2 = w_3 \neq e$ is cyclically normal. Let $a_1^{\alpha_1} \cdots a_k^{\alpha_k}$ be a normal form for $w_3$.
\newline \underline{\em Case 2.1}. $k = 1$.
\newline In this case, letting $a_k^{\alpha_k} = a^{\alpha}$, we have $g^{p} = w_1 a^{p\alpha} w^{-1}_1$, and so, arguing as in Case 1, we are done.
\newline \underline{\em Case 2.2}. $k > 1$. 
\newline In this case $g^{p}$ is spelled by the following normal form:
$$w_1 \underbrace{w_3 \cdots w_3}_{p} w^{-1}_1,$$
and so, clearly, we are done.
\newline \underline{\em Case 3}. $w_3 \neq e$ and $w_2 \neq e$.
\newline In this case, letting $w'_0$ stand for a normal form for $w_3 w'_2 w_2$, $g^{p}$ is spelled by the following normal form:
	$$g^{p} = w_1 w_2\underbrace{w'_0 \cdots w'_0}_{p-1}w_3w'_2 w^{-1}_1,$$
Furthermore, by item (3) and (5) of Fact \ref{bark_fact}, $sp(w'_0) = sp(w_3) \cup sp(w_2) = sp(w_3) \cup sp(w'_2) = sp(w_3) \cup sp(w_2) \cup sp(w'_2)$, and so we are done.
\end{proof}
	
	\begin{proposition}\label{prop_for_first_nec} Let $G = G(\Gamma, \mathfrak{p})$ and $g \in G$.
	\begin{enumerate}[(1)]
	\item If $a_1, a_2, b_1, b_2 \in \Gamma - sp(g)$ are distinct and $a_i$ is not adjacent to $b_i$ ($i = 1, 2$), then for every $n \geq 2$ the element $ga_1^{-1}a_2b_1^{-1}b_2$ has no $n$-th root.
	\item If $a, b_1, b_2, b_3, b_4 \in \Gamma$ are distinct, $a$ is not adjacent to $b_i$ ($i = 1, 2, 3, 4$), and $\{ b_1, b_2, b_3, b_4 \} \cap sp(g) = \emptyset$, then for every $n \geq 2$ the element $ga^{-1}b_1^{-1}b_2ab_3^{-1}b_4$ has no $n$-th root.
\end{enumerate}
\end{proposition}

	\begin{proof} We prove (1). Let $g_* = ga_1^{-1}a_2b_1^{-1}b_2$, $A = \{ a_2, b_2 \}$ and $\mathbf{p} = \mathbf{p}_A$ the homomorphism from Fact \ref{fact}. Then $\mathbf{p}(g_*) = a_2b_2$. Since $a_2$ is not adjacent to $b_2$, for every $n \geq 2$ the element $a_2b_2$ does not have an $n$-th root. As $\mathbf{p}_A$ is an homomorphism, we are done.

\smallskip 

\noindent We prove (2). Let $g_* = ga^{-1}b_1^{-1}b_2ab_3^{-1}b_4$, $A = \{ a, b_1, b_2, b_3, b_4 \}$ and $\mathbf{p} = \mathbf{p}_A$ the homomorphism from Fact \ref{fact}. There are two cases:
\newline \underline{\em Case 1}. $\mathbf{p}(g) = e$.
\newline  Then $\mathbf{p}(g_*) = a^{-1}b_1^{-1}b_2ab_3^{-1}b_4$. Since $a$ is not adjacent to $b_i$ ($i = 1, 2, 3, 4)$, for every $n \geq 2$ the element $a^{-1}b_1^{-1}b_2ab_3^{-1}b_4$ does not have an $n$-th root. 
\newline \underline{\em Case 2}. $\mathbf{p}(g) \neq e$.
\newline  Since $sp(\mathbf{p}(g)) \subseteq sp(g) \cap \{ a, b_1, b_2, b_3, b_4 \} \subseteq \{ a \}$ and $\mathbf{p}(g) \neq e$, we must have $sp(\mathbf{p}(g)) = \{ a \}$. Hence, $\mathbf{p}(g_*) = a^{\alpha}b_1^{-1}b_2ab_3^{-1}b_4$, for $\alpha \in \mathbb{Z} - \{ 0 \}$. Since $a$ is not adjacent to $b_i$ ($i = 1, 2, 3, 4)$, for every $n \geq 2$ the element $a^{\alpha}b_1^{-1}b_2ab_3^{-1}b_4$ does not have an $n$-th root. 
\end{proof}

\section{Negative Side}

In this section we show that conditions (a)-(d) of Theorem \ref{main_th} are necessary. Concerning conditions (a)-(c) we prove three separate lemmas: Lemmas \ref{lemma1}, \ref{lemma2} and \ref{lemma3}. Lemmas \ref{lemma2} and \ref{lemma3} are more general that needed for the proof of Theorem \ref{main_th}, and of independent interest. Concerning condition (d), it follows from Lemma \ref{lemma4} and Observation \ref{lemma5}, which are also more general that needed for our purposes.

 We denote the cyclic groups by $C_n, C_{\infty}$ (or $\mathbb{Z}_n, \mathbb{Z}_{\infty} = \mathbb{Z}$ in additive notation).

	\begin{lemma}\label{lemma1} Let $G = G(\Gamma, \mathfrak{p})$, with $|\Gamma| = 2^\omega$. Suppose that there does not exist a countable $A \subseteq \Gamma$ such that for every $a \in \Gamma$ and $a \neq b \in \Gamma - A$, $a$ is adjacent to $b$. Then $G$ does not admit a Polish group topology.
\end{lemma}

	\begin{proof} Suppose that $G = (\Gamma, \mathfrak{p})$ is as in the assumptions of the theorem, and that $G = (G, \mathfrak{d})$ is Polish. Then either of the following cases happens:
	\begin{enumerate}[(i)]
		\item in $\Gamma$ there are $\{ a_i : i < \omega_1 \}$ and $\{ b_i : i < \omega_1 \}$ such that if $i < j < \omega_1$, then $a_i \neq a_j$, $b_i \neq b_j$, $|\{ a_i, a_j, b_i, b_j \}| = 4$ and $a_i$ is not adjacent to $b_i$;
	\item in $\Gamma$ there are $a_*$ and $\{ b_i : i < \omega_1 \}$ such that if $i < j < \omega_1$, then $|\{ a_*, b_i, b_j \}| = 3$ and $a_*$ is not adjacent to $b_i$.
\end{enumerate}
\underline{\em Case 1}. There are $\{ a_i : i < \omega_1 \}$ and $\{ b_i : i < \omega_1 \}$ as in (i) above.
\newline Without loss of generality we can assume that all the $\{ a_i : i < \omega_1 \}$ have fixed color $k^*_1$ and all the $\{ b_i : i < \omega_1 \}$ have fixed color $k^{*}_2$, for some $k^*_1, k^*_2 \in \{ p^n : p \text{ prime and } 1 \leq n \} \cup \{ \infty \}$. Let
$p$ be a prime such that if $k^*_\ell \neq \infty$ then $p > k^*_\ell$, for $\ell = 1, 2$. Recalling Convention \ref{convention}, let $(\zeta_n)_{n < \omega} \in (0, 1)_{\mathbb{R}}^\omega$ be as in Fact \ref{771_fact} for $f \in \omega^{\omega}$ constantly $p + 10$.
Using Observation \ref{observation_prelim}, by induction on $n < \omega$, choose $(i_n = i(n), j_n = j(n))$ such that:
	\begin{enumerate}[(a)]
	\item if $m < n$, then $j_m < i_n$;
	\item $i_n < j_n < \omega_1$;
	\item $\mathfrak{d}(a_{j(n)}^{-1}a_{i(n)}, e), \mathfrak{d}(b_{j(n)}^{-1}b_{i(n)}, e) < \zeta_{n+8}$.
\end{enumerate}
Consider now the following set of equations:
	$$ \Delta = \{ x_n = (x_{n+1})^{p}h_n^{-1} : n < \omega \},$$
where $h_n = b_{i(n)}^{-1} b_{j(n)} a_{i(n)}^{-1} a_{j(n)}$. By (c) above and Fact \ref{771_fact}(1)(B) we have $\mathfrak{d}(h_n^{-1}, e) < \zeta_{n+1}$, and so by Fact \ref{771_fact}(2) the set $\Delta$ is solvable in $G$. Let $(g'_n)_{n < \omega}$ witness this. Let $A$ the set of vertices of color $k^*_1$ or $k^*_2$, $\mathbf{p} = \mathbf{p}_A$ the homomorphism from Fact \ref{fact} and let $g_n = \mathbf{p}(g'_n)$. Then for every $n < \omega$ we have:
$$
 G \models (g_{n+1})^{p} = g_nh_n,
$$
and so by Proposition \ref{fact_word} we have:
$$
 sp(g_n) \subseteq sp(g_0) \cup \{ b_{i(\ell)},  b_{j(\ell)}, a_{i(\ell)}, a_{j(\ell)}: \ell < n \}.
$$
Let $n < \omega$ be such that $sp(g_0) \cap \{ b_{i(n)},  b_{j(n)}, a_{i(n)}, a_{j(n)} \} = \emptyset$. Then:
$$(g_{n+1})^{p} = g_nb_{i(n)}^{-1} b_{j(n)} a_{i(n)}^{-1} a_{j(n)}  \text{ and } sp(g_n) \cap \{ b_{i(n)},  b_{j(n)}, a_{i(n)}, a_{j(n)} \} = \emptyset,$$
which contradicts Proposition \ref{prop_for_first_nec}(1).
\newline \underline{\em Case 2}. There are $a_*$ and $\{ b_i : i < \omega_1 \}$ as in (ii) above.
\newline Let $k^*_1 = \mathfrak{p}(a_*)$. Without loss of generality, we can assume that all the $\{ b_i : i < \omega_1 \}$ have fixed color $k^{*}_2$, for some $k^{*}_2 \in \{ p^n : p \text{ prime and } 1 \leq n \} \cup \{ \infty \}$. Let
$p$ be a prime such that if $k^*_\ell \neq \infty$ then $p > k^*_\ell$, for $\ell = 1, 2$. Let $(\zeta_n)_{n < \omega} \in (0, 1)_{\mathbb{R}}^\omega$ be as in Fact \ref{771_fact} for $\bar{g}_n = (a_*)$ (and so in particular $\ell(n) = 1$) and $f \in \omega^{\omega}$ constantly $p + 10$.
Using Observation \ref{observation_prelim}, by induction on $n < \omega$, choose $(i_n = i(n), j_n = j(n), i'_n = i'(n), j'_n = j'(n))$ such that:
	\begin{enumerate}[(a)]
	\item if $m < n$, then $j'_m < i_n$;
	\item $i_n < j_n < i'_n < j'_n < \omega_1$;
	\item $\mathfrak{d}(b_{j(n)}^{-1} b_{i(n)}, e), \mathfrak{d}(b_{j'(n)}^{-1} b_{i'(n)}, e) < \zeta_{n+8}$.
\end{enumerate}
Consider now the following set of equations:
	$$ \Delta = \{ x_n = (x_{n+1})^{p}h_n^{-1} : n < \omega \},$$
where $h_n = a^{-1}_*b_{i(n)}^{-1} b_{j(n)} a_*b_{i'(n)}^{-1} b_{j'(n)}$. By (c) above and Fact \ref{771_fact}(1)(B) we have $\mathfrak{d}(h_n^{-1}, e) < \zeta_{n+1}$, and so by Fact \ref{771_fact}(2) the set $\Delta$ is solvable in $G$. Let $(g'_n)_{n < \omega}$ witness this. Let $A$ be the set of vertices of color $k^{*}_1$ or $k^{*}_2$, $\mathbf{p} = \mathbf{p}_A$ the homomorphism from Fact \ref{fact} and let $g_n = \mathbf{p}(g'_n)$. Then for every $n < \omega$ we have:
$$
 G \models (g_{n+1})^{p} = g_nh_n,
$$
and so by Proposition \ref{fact_word} we have:
$$
 sp(g_n) \subseteq sp(g_0) \cup \{ a_*\} \cup \{ b_{i(\ell)},  b_{j(\ell)}, b_{i'(\ell)},  b_{j'(\ell)}: \ell < n \}.
$$
Let $n < \omega$ be such that $sp(g_0) \cap \{ b_{i(n)},  b_{j(n)}, b_{i'(\ell)},  b_{j'(\ell)}\} = \emptyset$. Then:
$$(g_{n+1})^{p} = g_na^{-1}_*b_{i(n)}^{-1} b_{j(n)} a_*b_{i'(n)}^{-1} b_{j'(n)}  \text{ and } sp(g_n) \cap \{ b_{i(n)},  b_{j(n)}, b_{i'(n)}^{-1} b_{j'(n)} \} = \emptyset,$$
which contradicts Proposition \ref{prop_for_first_nec}(2).
\end{proof}


	Recall that we denote the cyclic group of order $n$ by $C_n$.

	\begin{lemma}\label{lemma2} Let $G = G' \oplus G''$, with $G'' = \bigoplus_{n < \omega} G_n$, $G_n = \bigoplus_{\alpha < \lambda_n} C_{k(n)}$, $\aleph_0 < \lambda_n$, $k(n) = p_n^{t(n)}$, for $p_n$ prime and $1 \leq t(n)$, and the $k(n)$ pairwise distinct. Then $G$ does not admit a Polish group topology.
\end{lemma}

	\begin{proof}
Suppose that $G = (G, \mathfrak{d})$ is Polish and let $(\zeta_n)_{n < \omega} \in (0, 1)_{\mathbb{R}}^\omega$ be as in Fact \ref{771_fact} for $f \in \omega^{\omega}$ such that $f(n) = k(n) + 2$.
Assume that $G = G' \oplus G''$ is as in the assumptions of the lemma. 
	Without loss of generality we can assume that either of the following cases happens:
\begin{enumerate}[(i)]
	\item for every $n < m < \omega$, $p_n < p_m$;
	\item for every $n < \omega$, $p_n = p$ and $\prod_{i < n} p^{t(i)}$ is not divisible by $p^{t(n)}$.
\end{enumerate}
Using Observation \ref{observation_prelim}, by induction on $n < \omega$, choose $g_n, h_n \in G_n$ such that $g_n, h_n$ and $h_n^{-1}g_n$ have order $k(n)$ and $\mathfrak{d}(h_n^{-1}g_n, e) < \zeta_{n+1}$. Consider now the following set of equations:
	$$ \Gamma = \{ x_n = (x_{n+1})^{k(n)} h_n^{-1}g_n : n < \omega \}.$$
By Fact \ref{771_fact}(2) the set $\Gamma$ is solvable in $G$. Let $(d_n)_{n < \omega}$ witness this. Let then $n < \omega$ be such that $d_0 \in G' \oplus \bigoplus_{i < n} G_i$. Notice now that:
\[ \begin{array}{rcl}
	d_0 & = & (d_1)^{k(0)}h^{-1}_0g_0 \\
		& = & ((d_2)^{k(1)}h^{-1}_1g_1)^{k(0)}h^{-1}_0g_0 \\
		& = & (...((d_{n_{}+1})^{k(n)}h^{-1}_{n}g_{n})^{k(n-1)} \cdots h^{-1}_0g_0. 
\end{array}	\]
Let $\mathbf{p} = \mathbf{p}_{n}$ be the projection of $G$ onto $G_{n}$. Then we have:
$$G_{n} \models e = d_0 = (\mathbf{p}(d_{n+1})^{k(n)}h^{-1}_{n}g_{n})^{\prod_{i < n} k(i)} = (h^{-1}_{n}g_{n})^{\prod_{i < n} k(i)},$$
which is absurd.
\end{proof}


	When we write $G = \bigoplus_{\alpha < \lambda} \mathbb{Z}x_\alpha$ we mean that $x_{\alpha}$ is the generator of the $\alpha$-th copy of $\mathbb{Z}$. This convention is used in Lemmas \ref{lemma3} and \ref{lemma4}, and Observation \ref{lemma5}.
	
	\begin{lemma}\label{lemma3} Let $G = G_1 \oplus G_2$, with $G_2 = \bigoplus_{\alpha < \lambda} \mathbb{Z}x_\alpha$ and $\lambda > \aleph_0$. Then $G$ does not admit a Polish group topology.
\end{lemma}

	\begin{proof} Suppose that $G = (G, \mathfrak{d})$ is Polish and let $(\zeta_n)_{n < \omega} \in (0, 1)_{\mathbb{R}}^\omega$ be as in Fact \ref{771_fact} for $f \in \omega^{\omega}$ constantly $2 + 10$. Assume that $G = G_1 \oplus G_2$ is as in the assumptions of the lemma.
Using Observation \ref{observation_prelim}, by induction on $n < \omega$, choose $(i_n, j_n)$ such that:
	\begin{enumerate}[(i)]
	\item if $m < n$, then $j_m < i_n$;
	\item $i_n < j_n < \omega_1 \leq \lambda$;
	\item $\mathfrak{d}(x_{i_n}, x_{j_n}) < \zeta_{n+1}$.
	\end{enumerate}
	For every $n < \omega$ let:
	\begin{enumerate}[(a)]
	\item $x_{i_n} = h_n$;
	\item $x_{j_n} = g_n$;
	\item $\mathbb{Z}x_{i_n} \oplus \mathbb{Z}x_{j_n} = H_n$.
	\end{enumerate}
Consider now the following set of equations:
	$$ \Gamma = \{ x_n = (x_{n+1})^{2} h_n^{-1}g_n : n < \omega \}.$$
By Fact \ref{771_fact}(2) the set $\Gamma$ is solvable in $G$. Let $(d_n)_{n < \omega}$ witness this. Let then $n < \omega$ be such that $d_0 \in G_1 \oplus \bigoplus_{i < n} H_n$. Notice now that:
\[ \begin{array}{rcl}
	d_0 & = & (d_1)^{2}h^{-1}_0g_0 \\
		& = & ((d_2)^{2}h^{-1}_1g_1)^{2}h^{-1}_0g_0 \\
		& = & (...((d_{n_{}+1})^{2}h^{-1}_{n}g_{n})^{2} \cdots h^{-1}_0g_0. 
\end{array}	\]
Let $\mathbf{p}$ be the projection of $G$ onto $H_n$. Then we have:
$$H_{n} \models e = d_0 = (\mathbf{p}(d_{n+1})^{2}h^{-1}_{n}g_{n})^{2^n} = (h^{-1}_{n}g_{n})^{2^n},$$
which is absurd, since $H_n = \mathbb{Z}x_{i_n} \oplus \mathbb{Z}x_{j_n}$ is torsion-free and $h^{-1}_{n}g_{n} \neq e$.
\end{proof}

	In the rest of this section we use additive notation.
	
	\begin{lemma}\label{lemma4} Let $G = (G, \mathfrak{d})$ be an uncountable Polish group, $p$ a prime and $1 \leq t < \omega$. Suppose that $G = G_1 \oplus G_2$, with $G_2 = \bigoplus_{\alpha < \lambda} \mathbb{Z}_{p^t}x_{\alpha}$. If $\lambda > \aleph_0$, then there is $\bar{y} \subseteq G$ such that:
	\begin{enumerate}[(a)]
	\item $\bar{y} = (y_{\alpha} : \alpha < 2^{\aleph_0})$;
	\item $p^t y_{\alpha} = 0$ and, for $\ell < t$, $p^{\ell} y_{\alpha} \neq 0$;
	\item if $\alpha < \beta$, then $p^{t} (y_{\alpha} - y_{\beta}) = 0$, and, for $\ell < t$, $p^{\ell} (y_{\alpha} - y_{\beta}) \neq 0$;
	\item if $\alpha < \beta$, then $y_{\alpha} - y_{\beta}$ is not divisible by $p$ in $G$.
	\end{enumerate}
\end{lemma}

	\begin{proof} 
By induction on $n < \omega$, choose $(i_n, j_n)$ such that:
	\begin{enumerate}[(i)]
	\item if $m < n$, then $j_m < i_n$;
	\item $i_n < j_n < \omega_1$;
	\item $\mathfrak{d}(x_{i_n}, x_{j_n}) < 2^{-2^n}$.
	\end{enumerate}
	For $A \subseteq \omega$ and $n < \omega$, let:
	$$y_{A, n} = \sum \{ x_{i_k} - x_{j_k} : k \in A \text{, } k < n\}.$$
Then for every $A \subseteq \omega$, $(y_{A, n})_{n < \omega}$ is Cauchy. Let $y_A \in G$ be its limit. Then by continuity we have:
\begin{enumerate}[(a)]
	\item $p^{t} y_A = 0$, and, for $\ell < t$, $p^{\ell} y_A \neq 0$;
	\item if $A \neq B \subseteq \omega$, then $y_A$ and $y_B$ commute, $p^{t} (y_{A} - y_{B}) = 0$ and, for $\ell < t$, $p^{\ell} (y_{A} - y_{B}) \neq 0$.
\end{enumerate}
	We define the following equivalence relation $E$ on $\mathcal{P}(\omega)$: 
	$$A_1EA_2 \; \Leftrightarrow \; \exists x \in G(y_{A_1} - y_{A_2} = px).$$
We then have:
\begin{enumerate}[(I)]
	\item $E$ is analytic;
	\item if $B \subseteq \omega$, $n \in B$ and $A = B - \{ n \}$, then $\neg(y_A E y_B)$;
	\item by \cite[Lemma 13]{sh_for_CH} we have $|\mathcal{P}(\omega)/E| = 2^{\omega}$.
\end{enumerate}
	Hence, we can find $( y_{\alpha} : \alpha < 2^{\aleph_0})$ as wanted.
\end{proof}

\begin{observation}\label{lemma5} Let $G = (G, \mathfrak{d})$ be an uncountable Polish group, $p$ a prime and $1 \leq t < \omega$. Suppose that $G = G_0 \oplus G_1 \oplus G_2$, with $G_0$ countable, $G_1$ abelian, $\lambda > \aleph_0$ and $G_2 = \bigoplus_{\alpha < \lambda} \mathbb{Z}_{p^t}x_{\alpha}$. Let $(y_{\alpha} : \alpha < 2^{\aleph_0})$ be as in Lemma \ref{lemma4} with respect to the decomposition $G'_1 \oplus G'_2$ for $G'_1 = G_0 \oplus G_1$ and $G'_2 = G_2$. Then there is a pure embedding of $H = \bigoplus_{\alpha < 2^{\aleph_0}} \mathbb{Z}_{p^t} y_{\alpha}$ into the abelian group $G_1 \oplus G_2$.
\end{observation}

	\begin{proof} Define:
	$$\mathcal{U}_1 = \{ \alpha < 2^{\aleph_0} : \text{ for no } \xi \in \bigoplus_{\beta < \alpha} \mathbb{Z}_{p^t} y_{\beta} \text{ we have } y_{\alpha} - \xi \text{ is divisible by } p \text{ in } G_1 \oplus G_2\},$$
	$$\mathcal{U}_2 = \{ \alpha < 2^{\aleph_0} : \text{ for no } \xi \in \bigoplus_{\beta < \alpha} \mathbb{Z}_{p^t} y_{\beta} \text{ and } \ell < t \text{ we have } p^\ell(y_{\alpha} - \xi) = 0 \}.$$
Let $\mathcal{U} = \mathcal{U}_1 \cap \mathcal{U}_2$. For $\alpha \notin \mathcal{U}$, let $(\xi_{\alpha}, \ell_{\alpha})$ be witnesses of $\alpha \notin \mathcal{U}$, with $\ell = t$ if $\alpha \notin \mathcal{U}_1$.

\begin{claim} $|\mathcal{U}| = 2^{\aleph_0}$.
\end{claim}

\begin{claimproof} Suppose that $|\mathcal{U}| < 2^{\aleph_0}$ and let $\mu = \aleph_0 + |\mathcal{U}|$. Hence $\mathcal{U} \cap \mu^+$ is bounded. Let $\alpha_* = sup(\mathcal{U} \cap \mu^+)$. By Fodor's lemma for some stationary set $S \subseteq \mu^+ - (\alpha_* + 1)$ we have $\alpha \in S$ implies $(\xi_{\alpha}, \ell_{\alpha}) = (\xi_{*}, \ell_{*})$. Let $\alpha_1 < \alpha_2 \in S$. Then if $\ell_* = t$ we have that $y_{\alpha_2} - y_{\alpha_1}$ is divisible by $p$ in $G_1 \oplus G_2$, and if $\ell_* < t$ we have that $p^\ell(y_{\alpha_2} - y_{\alpha_1}) = 0$. In both cases we reach a contradiction, and so $|\mathcal{U}| = 2^{\aleph_0}$.
\end{claimproof}

\noindent Let now $\mathbf{p}_{\ell}$ be the canonical projection of $G$ onto $G_{\ell}$ ($\ell = 1, 2$). Then, by the claim above, $\{ (\mathbf{p}_{1} + \mathbf{p}_{2}) (y_\alpha): \alpha \in \mathcal{U} \}$ is a basis of a pure subgroup of $G_1 \oplus G_2$ isomorphic to $H$, and so we are done.
\end{proof}

\section{Positive Side}

	In this section we prove the the sufficiency of conditions (a)-(d) of Theorem \ref{main_th}.

	\begin{lemma}\label{positive} Suppose that $G = G(\Gamma, \mathfrak{p})$ satisfies conditions (a)-(d) of Theorem \ref{main_th} and $|\Gamma| = 2^{\omega}$. Then $G$ is realizable as the group of automorphisms of a countable structure.
\end{lemma}

	\begin{proof} Let $G = G(\Gamma, \mathfrak{p})$ be as in the assumptions of the theorem. Then we have:
	$$G \cong H \oplus \bigoplus_{p^n \mid n_*} \bigoplus_{\alpha < \lambda_{(p, n)}} \mathbb{Z}_{p^n},$$
for some countable group $H$, natural number $n_* < \omega$, and $\lambda_{(p, n)} \in \{ 0, 2^{\aleph_0} \}$ (here we are crucially using conditions (a)-(d) of the statement of the theorem, of course). Since finite sums of groups realizable as groups of automorphisms of countable structures are realizable as groups of automorphisms of countable structures, it suffices to show that for given $p^n$ the group:
$$H_1 = \bigoplus_{\alpha < 2^{\aleph_0}} \mathbb{Z}_{p^n} \cong \mathbb{Z}_{p^n}^{\omega}$$ is realizable as the group of automorphisms of countable structure. To this extent, let $A$ be a countable first-order structure such that $Aut(A) = \mathbb{Z}_{p^n}$. Let $B$ be the disjoint union of $\aleph_0$ copies of $A$, then $\mathbb{Z}_{p^n}^{\omega} \cong Aut(B)$, and so we are done.
\end{proof}

\end{document}